\documentclass[12pt]{amsart}

\usepackage{amsmath,amsfonts,amssymb,amsthm,amscd,epsfig,comment}
\usepackage[all]{xy}
\def\C{\mathbb{C}}

\def\Z{\mathbb{Z}}

\def\R{{\mathbb{R}}}
\def\id{\mathrm{id}}

\def\g{\ensuremath{\mathfrak{g}}}
\def\sl{\ensuremath{\mathfrak{sl}}}

\def\w{\mathbf{w}}

\def\T{\ensuremath{\mathfrak{T}}}
\def\ke{{\tilde e}}
\def\kf{{\tilde f}}

\def\cprime{$'$}

\DeclareMathOperator{\im}{im} %Image of a map
\DeclareMathOperator{\Hom}{Hom}

\DeclareMathOperator{\End}{End}

\DeclareMathOperator{\wt}{wt}
\DeclareMathOperator{\Span}{Span}

\DeclareMathOperator{\diag}{diag}

\DeclareMathOperator{\Ob}{Ob}
\DeclareMathOperator{\flip}{flip}

\newtheorem{theo}{Theorem}[section]
\newtheorem{prop}[theo]{Proposition}
\newtheorem{lem}[theo]{Lemma}

\newtheorem{defin}[theo]{Definition}
\newtheorem{example}[theo]{Example}
\newtheorem*{rem*}{Remark}

\numberwithin{equation}{section}

\allowdisplaybreaks

\begin{document}
\title{Braided and coboundary monoidal categories}
\author{Alistair Savage}
\address{University of Ottawa\\
Ottawa, Ontario \\ Canada} \email{alistair.savage@uottawa.ca}
\thanks{This research was supported by the Natural
Sciences and Engineering Research Council (NSERC) of Canada}
\date{April 29, 2008}

\dedicatory{Dedicated to Ivan Shestakov on his sixtieth birthday}

\begin{abstract}
We discuss and compare the notions of braided and coboundary
monoidal categories. Coboundary monoidal categories are analogues of
braided monoidal categories in which the role of the braid group is
replaced by the cactus group.  We focus on the categories of
representations of quantum groups and crystals and explain how while
the former is a braided monoidal category, this structure does not
pass to the crystal limit.  However, the categories of
representations of quantum groups of finite type also possess the
structure of a coboundary category which does behave well in the
crystal limit.  We explain this construction and also a recent
interpretation of the coboundary structure using quiver varieties.
This geometric viewpoint allows one to show that the category of
crystals is in fact a coboundary monoidal category for arbitrary
symmetrizable Kac-Moody type.
\end{abstract}

\maketitle \thispagestyle{empty}

%\tableofcontents

\section*{Introduction}

In this expository paper we discuss and contrast two types of
categories -- braided monoidal categories and coboundary monoidal
categories -- paying special attention to how the categories of
representations of quantum groups and crystals fit into this
framework.  Monoidal categories are essentially categories with a
tensor product, such as the categories of vector spaces, abelian
groups, sets and topological spaces. Braided monoidal categories are
well-studied in the literature. They are monoidal categories with an
action of the braid group on multiple tensor products.  The example
that interests us the most is the category of representations of a
quantum group $U_q(\g)$.  Coboundary monoidal categories are perhaps
less well known than their braided cousins. The concept is similar,
the difference being that the role of the braid group is now played
by the so-called \emph{cactus group}. A key component in the
definition of a coboundary monoidal category is the \emph{cactus
commutor}, which assumes the role of the braiding.

The theory of crystals can be thought of as the $q \to \infty$ (or
$q \to 0$) limit of the theory of quantum groups.  In this limit,
representations are replaced by combinatorial objects called crystal
graphs.  These are edge-colored directed graphs encoding important
information about the representations from which they come.
Developing concrete realizations of crystals is an active area of
research and there exist many different models.

It is interesting to ask if the structure of a braided monoidal
category passes to the crystal limit.  That is, does one have an
induced structure of a braided monoidal category on the category of
crystals.  The answer is no.  In fact, one can prove that it is
impossible to give the category of crystals the structure of a
braided monoidal category (see
Proposition~\ref{prop:crystals-not-braided}).  However, the
situation is more hopeful if one instead considers coboundary
monoidal categories. For quantum groups of finite type, there is a
way -- a unitarization procedure introduced by Drinfel\cprime d
\cite{Dri90} -- to use the braiding on the category of
representations to define a cactus commutor on this category.  This
structure passes to the crystal limit and one can define a
coboundary structure on the category of crystals in finite type (see
\cite{HK06,KT07}). Kamnitzer and Tingley \cite{KT06} gave an
alternative definition of the crystal commutor which makes sense for
quantum groups of arbitrary symmetrizable Kac-Moody type. However,
while this definition agrees with the previous one in finite type,
it is not obvious that it satisfies the desired properties, giving
the category of crystals the structure of a coboundary category, in
other types.

In \cite{Sav08a}, the author gave a geometric realization of the
cactus commutor using quiver varieties.  In this setting, the
commutor turns out to have a very simple interpretation -- it
corresponds to simply taking adjoints of quiver representations.
Equipped with this geometric description, one is able to show that
the crystal commutor satisfies the requisite properties and thus the
category of crystals, in arbitrary symmetrizable Kac-Moody type, is
a coboundary category.

In the current paper, when discussing the topics of quantum groups
and crystals, we will often restrict our attention to the quantum
group $U_q(\sl_2)$ and its crystals.  This allows us to perform
explicit computations illustrating the key concepts involved.  The
reader interested in the more general case can find the definitions
in the references given throughout the paper.

The organization of this paper is as follows.  In
Section~\ref{sec:groups}, we introduce the braid and cactus groups
that play an important role in the categories in which we are
interested.  In Sections~\ref{sec:monoidal}
through~\ref{sec:coboundary-monoidal}, we define monoidal, braided
monoidal, and coboundary monoidal categories.  We review the theory
of quantum groups and crystals in
Section~\ref{sec:quantum-groups-crystals}.  In
Section~\ref{sec:crystal-commutor} we recall the various definitions
of cactus commutors in the categories of representations of quantum
groups and crystals.  Finally, in Section~\ref{sec:geom-commutor},
we give the geometric interpretation of the commutor.

The author would like to thank J. Kamnitzer and P. Tingley for very
useful discussions during the writing of this paper and for helpful
comments on an earlier draft.

%%%%%%%%%%%%%%%%%%%%%%%%%%%%%%%%%%%%%%%%%%%%%%%%%%%%%%%%%%%%%%%%%%%%%%%%%%%%%

\section{The braid and cactus groups} \label{sec:groups}

\subsection{The braid group} \label{sec:braid-group}

\begin{defin}[Braid group]
For $n$ a positive integer, the \emph{$n$-strand Braid group}
$\mathcal{B}_n$ is the group with generators $\sigma_1, \dots,
\sigma_{n-1}$ and relations
\begin{enumerate}
\item $\sigma_i \sigma_j = \sigma_j \sigma_i$ for $|i-j| \ge 2$, and
\item $\sigma_i \sigma_{i+1} \sigma_i = \sigma_{i+1} \sigma_i
\sigma_{i+1}$ for $1 \le i \le n-2$.
\end{enumerate}
These relations are known as the \emph{braid relations} and the
second is often called the \emph{Yang-Baxter equation}.
\end{defin}

Recall that the symmetric group $S_n$ is the group on generators
$s_1, \dots, s_{n-1}$ satisfying the same relations as for the
$\sigma_i$ above in addition to the relations $s_i^2 =1$ for all $1
\le i \le n-1$.  We thus have a surjective group homomorphism
$\mathcal{B}_n \twoheadrightarrow S_n$.  The kernel of this map is
called the \emph{pure braid group}.

The braid group has several geometric interpretations.  The one from
which its name is derived is the realization in terms of braids.  An
\emph{$n$-strand braid} is an isotopy class of a union of $n$
non-intersecting smooth curve segments (strands) in $\R^3$ with end
points $\{1,2,\dots,n\} \times \{0\} \times \{0,1\}$, such that the
third coordinate is strictly increasing from 0 to 1 in each strand.
The set of all braids with multiplication giving by placing one
braid on top of another (and rescaling so that the third coordinate
ranges from 0 to 1) is isomorphic to $\mathcal{B}_n$ as defined
algebraically above.

The braid group is also isomorphic to the mapping class group of the
$n$-punctured disk -- the group of self-homeomorphisms of the
punctured disk with $n$-punctures modulo the subgroup consisting of
those homeomorphisms isotopic to the identity map.  One can picture
the isomorphism by thinking of each puncture being connected to the
boundary of the disk by a string.  Each homeomorphism of the
$n$-punctured disk can then be seen to yield a braiding of these
strings.  The pure braid group corresponds to the classes of
homeomorphisms that map each puncture to itself.

A similar geometric realization of the braid group is as the
fundamental group of the configuration space of $n$ points in the
unit disk $D$. A loop from one configuration to itself in this space
defines an $n$-strand braid where each strand is the trajectory in
$D \times [0,1]$ traced out by one of the $n$ points. If the points
are labeled, then we require each point to end where it started and
the corresponding fundamental group is isomorphic to the pure braid
group.

%%%%%%%%%%%%%%%%%%%%%%%%%%%%%%%%%%%%%%%%%%%%%%%%%%%%%%%%%%%%%%%%%%%%%%%%%%%%%

\subsection{The cactus group}

Fix a positive integer $n$.  For $1 \le p < q \le n$, let
\[
\hat s_{p,q} = \begin{pmatrix} 1 & \cdots & p-1 & p & p+1 & \cdots &
q & q+1 & \cdots & n \\ 1 & \cdots & p-1 & q & q-1 & \cdots & p &
q+1 & \cdots & n \end{pmatrix} \in S_n.
\]
Since $s_{i,i+1}=s_i$, these elements generate $S_n$. If $1 \le p <
q \le n$ and $1 \le k < l \le n$, we say that $p<q$ and $k<l$ are
\emph{disjoint} if $q < k$ or $l < p$. We say that $p<q$
\emph{contains} $k<l$ if $p \le k < l \le q$.

\begin{defin}[Cactus group]
For $n$ a positive integer, the \emph{$n$-fruit cactus group} $J_n$
is the group with generators $s_{p,q}$ for $1 \le p < q \le n$ and
relations
\begin{enumerate}
\item $s^2_{p,q} = 1$,
\item $s_{p,q} s_{k,l} = s_{k,l} s_{p,q}$ if $p<q$ and $k<l$ are
disjoint, and
\item $s_{p,q} s_{k,l} = s_{r,t} s_{p,q}$ if $p<q$ contains $k<l$,
where $r=\hat s_{p,q}(l)$ and $t=\hat s_{p,q}(k)$.
\end{enumerate}
\end{defin}
It is easily checked that the elements $\hat s_{p,q}$ of the
symmetric group satisfy the relations defining the cactus group and
thus the map $s_{p,q} \mapsto \hat s_{p,q}$ extends to a surjective
group homomorphism $J_n \twoheadrightarrow S_n$.  The kernel of this
map is called the \emph{pure cactus group}.

The cactus group also has a geometric interpretation.  In
particular, the kernel of the surjection $J_n \twoheadrightarrow
S_n$ is isomorphic to the fundamental group of the Deligne-Mumford
compactification $\overline{M}_0^{n+1}(\R)$ of the moduli space of
real genus zero curves with $n+1$ marked points.  The generator
$s_{p,q}$ of the cactus group corresponds to a path in
$\overline{M}_0^{n+1}(\R)$ in which the marked points $p, \dots, q$
balloon off into a new component, this components flips, and then
the component collapses, the points returning in reversed order.
Elements of $\overline{M}_0^{n+1}(\R)$ look similar to cacti of the
genus {\it Opuntia} (the marked points being flowers) which
justifies the name \emph{cactus group} (see
Figure~\ref{fig:cactus}).  Note the similarity with the last
geometric realization of the braid group mentioned in
Section~\ref{sec:braid-group}.  Both the pure braid group and the
pure cactus group are fundamental groups of certain spaces with
marked points. We refer the reader to \cite{DJS03,Dev99,HK06} for
further details on this aspect of the cactus group.
\begin{figure}
\centering
\includegraphics[width=3.5cm]{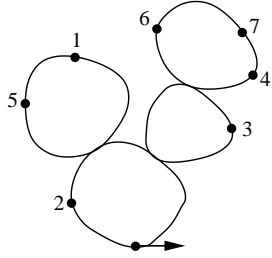}
\caption{A 7-fruited cactus} \label{fig:cactus}
\end{figure}

%%%%%%%%%%%%%%%%%%%%%%%%%%%%%%%%%%%%%%%%%%%%%%%%%%%%%%%%%%%%%%%%%%%%%%%%%%%%%

\subsection{The relationship between the braid and cactus groups}

The relationship between the braid group and the cactus group is not
completely understood.  While the symmetric group is a quotient of
both, neither is a quotient of the other.  However, there is a
homomorphism from the cactus group into the pro-unipotent completion
of the braid group (see the proof of Theorem~3.14 in \cite{EHKR}).
This map is closely related to the unitarization procedure of
Drinfel\cprime d to be discussed in
Section~\ref{sec:Drinfeld-unitarization}.

In the next few sections, we will define categories that are closely
related to the braid and cactus groups.  We will see that there are
some connections between the two.  In addition to the unitarization
procedure, we will see that braidings satisfy the so-called cactus
relation (see Proposition~\ref{prop:braiding-cactus}).

%%%%%%%%%%%%%%%%%%%%%%%%%%%%%%%%%%%%%%%%%%%%%%%%%%%%%%%%%%%%%%%%%%%%%%%%%%%%%

\section{Monoidal categories} \label{sec:monoidal}

\subsection{Definitions}

Recall that for two functors $F, G: \mathcal{C} \to \mathcal{D}$ a
\emph{natural transformation} $\varphi : F \to G$ is a collection of
morphisms $\varphi_U : F(U) \to G(U)$, $U \in \Ob \mathcal{C}$, such
that for all $f \in \Hom_{\mathcal{C}}(U,V)$, we have $\varphi_V
\circ F(f) = G(f) \circ \varphi_U : F(U) \to G(V)$.  If the maps
$\varphi_U$ are all isomorphisms, we call $\varphi$ a \emph{natural
isomorphism}.  We will sometimes refer to the $\varphi_U$ themselves
as \emph{natural isomorphisms} when the functors involved are clear.

\begin{defin}[Monoidal Category]
A \emph{monoidal category} is a category $\mathcal{C}$ equipped with
the following:
\begin{enumerate}
\item a bifunctor $\otimes : \mathcal{C} \times \mathcal{C} \to
\mathcal{C}$ called the \emph{tensor product},

\item natural isomorphisms (the \emph{associator})
\[
\alpha_{U,V,W} : (U \otimes V) \otimes W
\stackrel{\cong}{\longrightarrow} U \otimes (V \otimes W)
\]
for all $U,V,W \in \Ob \mathcal{C}$ satisfying the \emph{pentagon
axiom}: for all $U,V,W,X \in \Ob \mathcal{C}$, the diagram
\[
\xymatrix{ & ((U \otimes V) \otimes W) \otimes X
\ar[dl]_{\alpha_{U,V,W} \otimes \id_X} \ar[dr]^{\alpha_{U \otimes V,
W, X}} &
\\ (U \otimes (V \otimes W)) \otimes X \ar[d]^{\alpha_{U,V \otimes W, X}}
& & (U \otimes V) \otimes (W \otimes X) \ar[d]_{\alpha_{U,V,W \otimes X}} \\
U \otimes ((V \otimes W) \otimes X) \ar[rr]^{\id_U \otimes
\alpha_{V,W,X}} & & U \otimes (V \otimes (W \otimes X))}
\]
commutes, and

\item a \emph{unit object} $\mathbf{1} \in \Ob \mathcal{C}$ and
natural isomorphisms
\[
\lambda_V : \mathbf{1} \otimes V \stackrel{\cong}{\longrightarrow}
V, \quad \rho_V : V \otimes \mathbf{1}
\stackrel{\cong}{\longrightarrow} V
\]
for every $V \in \Ob \mathcal{C}$, satisfying the \emph{triangle
axiom}: for all $U,V \in \Ob \mathcal{C}$, the diagram
\[
\xymatrix{ (U \otimes \mathbf{1}) \otimes V
\ar[rr]^{\alpha_{U,\mathbf{1},V}} \ar[dr]^{\rho_U \otimes \id_V} & &
U
\otimes (\mathbf{1} \otimes V) \ar[dl]^{\id_U \otimes \lambda_V} \\
& U \otimes V & }
\]
commutes.
\end{enumerate}
A monoidal category is said to be \emph{strict} if we can take
$\alpha_{U,V,W}$, $\lambda_V$ and $\rho_V$ to be identity morphisms
for all $U,V,W \in \Ob \mathcal{C}$.  That is, a \emph{strict
monoidal category} is one in which
\[
V \otimes \mathbf{1} = V,\quad \mathbf{1} \otimes V = V,\quad (U
\otimes V) \otimes W = U \otimes (V \otimes W)
\]
for all $U,V,W \in \Ob \mathcal{C}$.
\end{defin}

The MacLane Coherence Theorem \cite[{\S}VII.2]{Mac98} states that
the pentagon and triangle axioms ensure that any for any two
expressions obtained from $V_1 \otimes V_2 \otimes \dots \otimes
V_n$ by inserting $\mathbf{1}$'s and parentheses, all isomorphisms
of these two expressions consisting of compositions of $\alpha$'s,
$\lambda$'s and $\rho$'s are equal.  This condition is called the
\emph{associativity axiom}.  In a monoidal category, we can use the
natural isomorphisms to identify all expressions of the above type
and so we often write multiple tensor products without brackets. In
fact, every monoidal category is equivalent to a strict one
\cite[{\S}XI.3]{Mac98}.

%%%%%%%%%%%%%%%%%%%%%%%%%%%%%%%%%%%%%%%%%%%%%%%%%%%%%%%%%%%%%%%%%%%%%%%%%%%%%

\subsection{Examples} \label{sec:monoidal-examples}

Most of the familiar tensor products yield monoidal categories. For
instance, for a commutative ring $R$, the category of $R$-modules is
a monoidal category.  We have the usual tensor product $A \otimes_R
B$ of modules $A$ and $B$. The unit object is $R$ and we have the
natural isomorphisms
\begin{gather*}
\alpha : A \otimes_R (B \otimes_R C) \cong (A \otimes_R B) \otimes_R
C,\quad
\alpha(a \otimes_R (b \otimes_R c)) = (a \otimes_R b) \otimes_R c \\
\lambda : R \otimes_R A \cong A,\quad \lambda(r \otimes_R a) =
ra,\\
\rho : A \otimes_R R \cong A,\quad \rho(a \otimes_R r) = ra.
\end{gather*}
In particular, the categories of abelian groups (where $R=\Z$) and
vector spaces (where $R$ is a field) are monoidal categories.  In a
similar fashion, the category of $R$-algebras is monoidal under the
usual tensor product of algebras. For an arbitrary (not necessarily
commutative) ring $R$, the category of $R$-$R$ bimodules is also
monoidal under $\otimes_R$. The categories of sets and topological
spaces with the cartesian product are monoidal categories with unit
objects 1 (the set with a single element) and $*$ (the
single-element topological space) respectively.

%%%%%%%%%%%%%%%%%%%%%%%%%%%%%%%%%%%%%%%%%%%%%%%%%%%%%%%%%%%%%%%%%%%%%%%%%%%%%

\section{Braided monoidal categories} \label{sec:braided-monoidal}

\subsection{Definitions}

\begin{defin}[Braided monoidal category]
A \emph{braided monoidal category} (or \emph{braided tensor
category}) is a monoidal category $\mathcal{C}$ equipped with
natural isomorphisms $\sigma_{U,V} : U \otimes V \to V \otimes U$
for all $U,V \in \Ob \mathcal{C}$ satisfying the \emph{hexagon}
axiom:  for all $U,V,W \in \mathcal{C}$, the diagrams
\[
\xymatrix{ & U \otimes (V \otimes W) \ar[rr]^{\sigma_{U,V
\otimes W}} & & (V \otimes W) \otimes U \ar[dr]^{\alpha_{V,W,U}} & \\
(U \otimes V) \otimes W \ar[ur]^{\alpha_{U,V,W}}
\ar[dr]^{\sigma_{U,V} \otimes \id_W} & & & & V \otimes (W \otimes U)
\\
& (V \otimes U) \otimes W \ar[rr]^{\alpha_{V,U,W}} & & V \otimes (U
\otimes W) \ar[ur]^{\id_V \otimes \sigma_{U,W}} & }
\]
and
\[
\xymatrix{ & U \otimes (V \otimes W) \ar[rr]^{\sigma^{-1}_{V
\otimes W,U}} & & (V \otimes W) \otimes U \ar[dr]^{\alpha_{V,W,U}} & \\
(U \otimes V) \otimes W \ar[ur]^{\alpha_{U,V,W}}
\ar[dr]^{\sigma^{-1}_{V,U} \otimes \id_W} & & & & V \otimes (W
\otimes U)
\\
& (V \otimes U) \otimes W \ar[rr]^{\alpha_{V,U,W}} & & V \otimes (U
\otimes W) \ar[ur]^{\id_V \otimes \sigma^{-1}_{W,U}} & }
\]
commute.  The collection of maps $\sigma_{U,V}$ is called a
\emph{braiding}.
\end{defin}

For a braided monoidal category $\mathcal{C}$ and $U,V,W \in \Ob
\mathcal{C}$, consider the following diagram where we have omitted
bracketings and associators (or assumed that $\mathcal{C}$ is
strict).
\[
\xymatrix{ & V \otimes U \otimes W  \ar[dr]^{\id_V \otimes \sigma_{U,W}} & \\
U \otimes V \otimes W \ar[ur]^{\sigma_{U,V} \otimes \id_W}
\ar[rr]^{\sigma_{U,V \otimes W}} \ar[d]^{\id_U \otimes \sigma_{V,W}}
& & V \otimes W \otimes U \ar[d]^{\sigma_{V,W} \otimes \id_U} \\
U \otimes W \otimes V \ar[rr]^{\sigma_{U,W \otimes V}}
\ar[dr]_{\sigma_{U,W} \otimes \id_V} & & W \otimes V \otimes U \\
& W \otimes U \otimes V \ar[ur]_{\id_W \otimes \sigma_{U,V}}}
\]
The top and bottom triangles commute by the hexagon axiom and the
middle rectangle commutes by the naturality of the braiding (that
is, by the fact that it is a natural isomorphism).  Therefore, if we
write $\sigma_1$ for the map $\sigma \otimes \id$ and $\sigma_2$ for
the map $\id \otimes \sigma$, we have
\[
\sigma_1 \sigma_2 \sigma_1 = \sigma_2 \sigma_1 \sigma_2,
\]
the Yang-Baxter relation for $\mathcal{B}_3$.  It follows that in a
braided monoidal category, the braid group $\mathcal{B}_n$ acts on
$n$-fold tensor products. That is, if we denote $\id^{\otimes (i-1)}
\otimes \sigma \otimes \id^{\otimes (n-i-1)}$ by $\sigma_i$, then a
composition of such maps depends only on the corresponding element
of the braid group.

\begin{defin}[Symmetric monoidal category]
A \emph{symmetric monoidal category} is a braided monoidal category
$\mathcal{C}$ where $\sigma_{V,U} \circ \sigma_{U,V} = \id_{U
\otimes V}$ for all $U, V \in \Ob \mathcal{C}$.
\end{defin}
In any symmetric monoidal category, the symmetric group $S_n$ acts
on $n$-fold tensor products in the same way that the braid group
acts in a braided monoidal category. We note that there is
conflicting terminology in the literature. For instance, some
authors refer to monoidal categories as tensor categories while
others (see, for instance, \cite{CP95}) refer to symmetric monoidal
categories as tensor categories and braided monoidal categories as
quasitensor categories.

%%%%%%%%%%%%%%%%%%%%%%%%%%%%%%%%%%%%%%%%%%%%%%%%%%%%%%%%%%%%%%%%%%%%%%%%%%%%%

\subsection{Examples}

Many of the examples in Section~\ref{sec:monoidal-examples} can in
fact be given the structure of a symmetric monoidal category.  In
particular, the categories of $R$-modules and $R$-algebras over a
commutative ring $R$, the category of sets, and the category of
topological spaces are all symmetric monoidal categories.  In all of
these examples, the braiding is given by $\sigma(a \otimes b) = b
\otimes a$.

An example that shall be especially important to us is the category
of representations of a quantum group.  It can be given the
structure of a braided monoidal category but is not a symmetric
monoidal category.

%%%%%%%%%%%%%%%%%%%%%%%%%%%%%%%%%%%%%%%%%%%%%%%%%%%%%%%%%%%%%%%%%%%%%%%%%%%%%

\section{Coboundary monoidal categories}
\label{sec:coboundary-monoidal}

\subsection{Definitions}
Coboundary monoidal categories are analogues of braided monoidal
categories in which the role of the braid group is replaced by the
cactus group.  As we shall see, they are better suited to the theory
of crystals than braided monoidal categories.

\begin{defin}[Coboundary monoidal category]
\label{def:coboundary-cat} A \emph{coboundary monoidal category} is
a monoidal category $\mathcal{C}$ together with natural isomorphisms
$\sigma^c_{U,V} : U \otimes V \to V \otimes U$ for all $U,V \in \Ob
\mathcal{C}$ satisfying the following conditions:
\begin{enumerate}
\item $\sigma^c_{V,U} \circ \sigma^c_{U,V} = \id_{U \otimes V}$, and
\item \emph{the cactus relation}: for all $U,V,W \in \Ob \mathcal{C}$, the diagram
\begin{equation} \label{eq:cactus-relation}
\xymatrix{ U \otimes V \otimes W \ar[rr]^{\sigma^c_{U,V} \otimes
\id_W} \ar[d]_{\id_U \otimes \sigma^c_{V,W}} & & V \otimes U \otimes
W \ar[d]^{\sigma^c_{V \otimes U,W}}\\
U \otimes W \otimes V \ar[rr]^{\sigma^c_{U,W \otimes V}} & & W
\otimes V \otimes U}
\end{equation}
commutes.
\end{enumerate}
The collection of maps $\sigma^c_{U,V}$ is called a \emph{cactus
commutor}.
\end{defin}
We will use the term \emph{commutor} for a collection of natural
isomorphisms $\sigma_{U,V} : U \otimes V \to V \otimes U$ for all
objects $U, V \in \Ob \mathcal{C}$ (this is sometimes called a
\emph{commutativity constraint}), reserving the term \emph{cactus
commutor} for a commutor satisfying the conditions in the definition
above.

Suppose $\mathcal{C}$ is a coboundary category and $U_1, \dots, U_n
\in \Ob \mathcal{C}$.  For $1 \le p < q \le n$, one defines natural
isomorphisms
\begin{align*}
\sigma^c_{p,q}&=(\sigma^c_{p,q})_{U_1,\dots,U_n} : U_1 \otimes \dots
\otimes U_n \\
& \qquad \qquad \to U_1 \otimes \dots \otimes U_{p-1} \otimes
U_{p+1} \otimes \dots \otimes U_q \otimes U_p \otimes U_{q+1}
\otimes \dots \otimes U_n,\\
\sigma^c_{p,q} &:= \id_{U_1 \otimes \dots \otimes U_{p-1}} \otimes
\sigma^c_{U_p, U_{p+1} \otimes \dots \otimes U_q} \otimes
\id_{U_{q+1} \otimes \dots \otimes U_n}.
\end{align*}
For $1 \le p < q \le n$, one then defines natural isomorphisms
\[
s_{p,q} : U_1 \otimes \dots \otimes U_n \to U_1 \otimes U_2 \otimes
\dots \otimes U_{p-1} \otimes U_q \otimes U_{q-1} \otimes \dots U_p
\otimes U_{q+1} \otimes U_{q+2} \otimes \dots \otimes U_n
\]
recursively as follows.  Define $s_{p,p+1} = \sigma^c_{p,p+1}$ and
$s_{p,q} = \sigma^c_{p,q} \circ s_{p+1,q}$ for $q > p+1$.  We also
set $s_{p,p} = \id$.  The following proposition was proved by
Henriques and Kamnitzer.

\begin{prop}[{\cite[Lemma~3, Lemma~4]{HK06}}]
If $\mathcal{C}$ is a coboundary category and the natural
isomorphisms $s_{p,q}$ are defined as above, then
\begin{enumerate}
\item $s_{p,q} \circ s_{p,q} = \id$,
\item $s_{p,q} \circ s_{k,l} = s_{k,l} \circ s_{p,q}$ if $p<q$ and
$k < l$ are disjoint, and
\item $s_{p,q} \circ s_{k,l} = s_{r,t} \circ s_{p,q}$ if $p<q$ contains $k<l$,
where $r=\hat s_{p,q}(l)$ and $t=\hat s_{p,q}(k)$.
\end{enumerate}
\end{prop}
Therefore, in a coboundary monoidal category the cactus group $J_n$
acts on $n$-fold tensor products.  This is analogous to the action
of the braid group in a braided monoidal category.  For this reason,
the authors of \cite{HK06} propose the name \emph{cactus category}
for a coboundary monoidal category.

\begin{prop} \label{prop:braiding-cactus}
    Any braiding satisfies the cactus relation.
\end{prop}
\begin{proof}
    Suppose $\mathcal{C}$ is a braided monoidal category with
    braiding $\sigma$ and $U,V,W \in \Ob \mathcal{C}$.  By the
    axioms of a braided monoidal category, we have
    \[
        \sigma_{V \otimes U,W}^{-1} = (\id_V \otimes
        \sigma_{U,W}^{-1})(\sigma_{V,W}^{-1} \otimes \id_U)
    \]
    Taking inverses gives
    \[
        \sigma_{V \otimes U, W} = (\sigma_{V,W} \otimes \id_U)(\id_V
        \otimes \sigma_{U,W}).
    \]
    Therefore
    \begin{align*}
        (\sigma_{U,W \otimes V})(\id_U \otimes \sigma_{V,W}) &=
        (\id_W \otimes \sigma_{U,V})(\sigma_{U,W} \otimes \id_V)
        (\id_U \otimes \sigma_{V,W}) \\
        &= (\sigma_{V,W} \otimes \id_U)(\id_V \otimes \sigma_{U,W})
        (\sigma_{U,V} \otimes \id_W) \\
        &= (\sigma_{V \otimes U,W})(\sigma_{U,V} \otimes \id_W).
    \end{align*}
    The first equality uses the definition of a braiding and the second
    is the Yang-Baxter equation.
\end{proof}

Note that Proposition~\ref{prop:braiding-cactus} does not imply that
every braided monoidal category is a coboundary monoidal category
because we require cactus commutors to be involutions whereas
braidings, in general, are not.  In fact, any commutor that is both
a braiding and a cactus commutor is in fact a symmetric commutor
(that is, it endows the category in question with the structure of a
symmetric monoidal category).

\subsection{Examples}
The definition of coboundary monoidal categories was first given by
Drinfel\cprime d in \cite{Dri90}.  The name was inspired by the fact
that the representation categories of coboundary Hopf algebras are
coboundary monoidal categories.  Since the cactus group surjects
onto the symmetric group, any symmetric monoidal category is a
coboundary monoidal category.  Our main example of coboundary
categories which are not symmetric monoidal categories will be the
categories of representations of quantum groups and crystals.
Furthermore, we will see that the category of crystals cannot be
given the structure of a braided monoidal category.  Thus there
exist examples of coboundary monoidal categories that are not
braided.

%%%%%%%%%%%%%%%%%%%%%%%%%%%%%%%%%%%%%%%%%%%%%%%%%%%%%%%%%%%%%%%%%%%%%%%%%%%%%

\section{Quantum groups and crystals}
\label{sec:quantum-groups-crystals}

\subsection{Quantum groups} \label{sec:quantum-groups}
Compact groups and semisimple Lie algebras are rigid objects in the
sense that they cannot be deformed.  However, if one considers the
group algebra or universal enveloping algebra instead, a deformation
is possible. Such a deformation can be carried out in the category
of (noncommutative, noncocommutative) Hopf algebras. These
deformations play an important role in the study of the quantum
Yang-Baxter equation and the quantum inverse scattering method.
Another benefit is that the structure of the deformations and their
representations becomes more rigid and the concepts of canonical
bases and crystals emerge.

We introduce here the quantum group, or quantized enveloping
algebra, defined by Drinfel\cprime d and Jimbo.  For further details
we refer the reader to the many books on the subject (e.g.
\cite{CP95,HK,L93}). Let $\g$ be a Kac-Moody algebra with
symmetrizable generalized Cartan matrix $A = (a_{ij})_{i,j \in I}$
and symmetrizing matrix $D = \diag (s_i \in \Z_{> 0}\ |\ i \in I)$.
Let $P, P^\vee$, and $Q_+$ be the weight lattice, coweight lattice
and positive root lattice respectively.  Let $\C_q$ be the field
$\C(q^{1/2})$ where $q$ is a formal variable. For $n \in \Z$ and any
symbol $x$, we define
\begin{gather*}
    [n]_x = \frac{x^n - x^{-n}}{x-x^{-1}},\quad [0]_x! = 1,\quad [m]_x! =
    [m]_x [m-1]_x \cdots [1]_x \text{ for } m \in \Z_{>0},\\
    \begin{bmatrix} k \\ l \end{bmatrix}_x = \frac{[k]_x!}{[l]_x!
    [k-l]_x!}\ \text{ for } k,l \in \Z_{\ge 0}.
\end{gather*}

\begin{defin}[Quantum group $U_q(\g)$]
The \emph{quantum group} or \emph{quantized enveloping algebra}
$U_q(\g)$ is the unital associative algebra over $\C_q$ with
generators $e_i$, $f_i$ ($i \in I$) and $q^h$ ($h \in P^\vee$) with
defining relations
\begin{enumerate}
    \item $q^0=1$, $q^h q^{h'} = q^{h+h'}$ for $h,h' \in P^\vee$,
    \item $q^h e_i q^{-h} = q^{\alpha_i(h)} e_i$ for $h \in P^\vee$,
    \item $q^h f_i q^{-h} = q^{-\alpha_i(h)} f_i$ for $h \in P^\vee$,
    \item $e_i f_j - f_j e_i = \delta_{ij} {\displaystyle \frac{q^{s_i h_i} - q^{-s_i
    h_i}}{q^{s_i} - q^{-s_i}}} \text{ for } i,j \in I,$
    \item $\sum_{k=0}^{1-a_{ij}} (-1)^k \begin{bmatrix} 1-a_{ij} \\ k
    \end{bmatrix}_{q^{s_i}} e_i^{1-a_{ij}-k} e_j e_i^k = 0$ for $i \ne j$,
    \item $\sum_{k=0}^{1-a_{ij}} (-1)^k \begin{bmatrix} 1-a_{ij} \\ k
    \end{bmatrix}_{q^{s_i}} f_i^{1-a_{ij}-k} f_j f_i^k = 0$ for $i \ne
    j$.
\end{enumerate}
\end{defin}
As $q \to 1$, the defining relations for $U_q(\g)$ approach the
usual relations for $\g$ in the following sense. Taking the
``derivative'' with respect to $q$ in the second and third relations
gives
\begin{gather*}
    h q^{h-1} e_i q^{-h} + q^h e_i \left(-h q^{-h-1} \right) = \alpha_i(h)
    q^{\alpha_i(h)-1} e_i
    \stackrel{q \to 1}{\longrightarrow} he_i - e_ih = [h,e_i] = \alpha_i(h) e_i,\\
    h q^{h-1} f_i q^{-h} + q^h f_i \left(-h q^{-h-1} \right) = - \alpha_i(h)
    q^{-\alpha_i(h)-1} f_i
    \stackrel{q \to 1}{\longrightarrow} hf_i - f_ih = [h,f_i] = -\alpha_i(h) f,
\end{gather*}
Furthermore, if we naively apply L'H{\^ o}pital's rule, we have
\[
    \lim_{q \to 1} \frac{q^{s_i h_i} - q^{-s_i h_i}}{q^{s_i} - q^{-s_i}}
    = \lim_{q \to 1} \frac{s_i h_iq^{s_i h_i-1} + s_i h_iq^{-s_i h_i-1}}{s_i q^{s_i-1}
    + s_i q^{-s_i - 1}}
    = \frac{2s_i h_i}{2s_i} = h_i,
\]
and so the fourth defining relation of $U_q(\g)$ becomes
$[e_i,f_j]=\delta_{ij} h_i$ in the $q \to 1$ limit -- called the
\emph{classical limit}.  Similarly, since we have
\[
[n]_{q^{s_i}} \to n,\text{ and } \begin{bmatrix} 1 - a_{ij} \\ k
\end{bmatrix}_{q^{s_i}} \to \begin{pmatrix} 1 - a_{ij} \\ k
\end{pmatrix} \text{ as } q \to \infty,
\]
the last two relations (called the \emph{quantum Serre relations})
become the usual Serre relations in the classical limit. Thus, we
can think of $U_q(\g)$ as a deformation of $\g$. For a more rigorous
treatment of the classical limit, we refer the reader to
\cite[\S3.4]{HK}.

The algebra $U_q(\g)$ has a Hopf algebra structure given by
comultiplication
\[
    \Delta(q^h) = q^h \otimes q^h,\quad \Delta(e_i)
    = e_i \otimes q^{s_ih_i} + 1 \otimes e_i,\quad \Delta(f_i) = f_i \otimes 1
    + q^{-s_ih_i} \otimes f_i,
\]
counit
\[
    \varepsilon(q^h)=1,\quad \varepsilon(e_i)=\varepsilon(f_i)=0,
\]
and antipode
\[
    \gamma(q^h)=q^{-h},\quad \gamma(e_i)=-e_i q^{-s_ih_i},\quad
    \gamma(f_i)=-q^{s_i h_i} f_i.
\]
There are other choices but we will use the above in what follows.

The representations of $U(\g)$ can be $q$-deformed to
representations of $U_q(\g)$ in such a way that the dimensions of
the weight spaces are invariant under the deformation (see
\cite{HK,Lus88}). The $q$-deformed notion of a weight space is as
follows: for a $U_q(\g)$-module $M$ and $\lambda \in P$, the
$\lambda$-weight space of $M$ is
\[
    M^\lambda = \{v \in M\ |\ q^h v=q^{\lambda(h)} v \ \forall\ h \in P^\vee\}.
\]

%%%%%%%%%%%%%%%%%%%%%%%%%%%%%%%%%%%%%%%%%%%%%%%%%%%%%%%%%%%%%%%%%%%%%

\subsection{Crystal bases}
In this section we introduce the theory of crystal bases, which can
be thought of as the $q \to \infty$ limit of the representation
theory of quantum groups. In this limit, representations are
replaced by combinatorial objects called crystal graphs.  These
objects, which are often much easier to compute with than the
representations themselves, can be used to obtain such information
as dimensions of weight spaces (characters) and the decomposition of
tensor products into sums of irreducible representations.  For
details, we refer the reader to \cite{HK}.  We note that in
\cite{HK}, the limit $q \to 0$ is used. This simply corresponds to a
different choice of Hopf algebra structure on $U_q(\g)$.  We choose
to consider $q \to \infty$ to match the choices of \cite{CP95}.

For $n \in \Z_{\ge 0}$ and $i \in I$, define the \emph{divided
powers}
\[
    e_i^{(n)} = e_i^n/[n]_q!,\quad f_i^{(n)} = f_i^n/[n]_q!
\]
Let $M$ be an integrable $U_q(\g)$-module and let $M^\lambda$ be the
$\lambda$-weight space for $\lambda \in P$.  For $i \in I$, any
weight vector $u \in M^\lambda$ can be written uniquely in the form
\[
    u = \sum_{n=0}^\infty f_i^{(n)} u_n, \quad u_n \in \ker e_i \cap
    M^{\lambda+n\alpha_i}.
\]
Define the \emph{Kashiwara operators} $\ke_i, \kf_i : M \to M$ by
\[
    \ke_i u = \sum_{n=1}^\infty f_i^{(n-1)} u_n,\quad \kf_i u =
    \sum_{n=0}^\infty f_i^{(n+1)} u_n.
\]
Let $A$ be the integral domain of all rational functions in $\C_q$
that are regular at $q=\infty$.  That is, $A$ consists of all
rational functions that can be written in the form
$g_1(q^{-1/2})/g_2(q^{-1/2})$ for $g_1(q^{-1/2})$ and
$g_2(q^{-1/2})$ polynomials in $\C[q^{-1/2}]$ with
$g_2(q^{-1/2})|_{q^{-1/2}=0} \ne 0$ (one should think of these as
rational functions whose limit exists as $q \to \infty$).

\begin{defin}[Crystal basis]
A \emph{crystal basis} of a $U_q(\g)$-module $M$ is a pair $(L,B)$
such that
\begin{enumerate}
    \item L is a free $A$-submodule of $M$ such that $M = \C_q \otimes_A L$,
    \item $B$ is a $\C$-basis of the vector space $L/q^{-1/2}L$ over $\C$,
    \item $L = \bigoplus_\lambda L^\lambda$, $B = \bigsqcup_\lambda B^\lambda$
    where $L^\lambda = L \cap M^\lambda$, $B^\lambda = B \cap
    L^\lambda/q^{-1/2}L^\lambda$,
    \item $\ke_i L \subseteq L$, $\kf_i L \subseteq L$ for all $i
    \in I$,
    \item $\ke_i B \subseteq B \cup \{0\}$, $\kf_i B \subseteq B \cup \{0\}$
    for all $i \in I$, and
    \item for all $b,b' \in B$ and $i \in I$, $\ke_i b = b'$ if and only if
    $\kf_i b' = b$.
\end{enumerate}
\end{defin}
It was shown by Kashiwara \cite{K91} that all $U_q(\g)$-modules in
the category $\mathcal{O}^q_{\text{int}}$ (integrable modules with
weight space decompositions and weights lying in a union of sets of
the form $\lambda - Q_+$ for $\lambda \in P$) have unique crystal
bases (up to isomorphism).

A crystal basis can be represented by a \emph{crystal graph}.  The
crystal graph corresponding to a crystal basis $(L,B)$ is an
edge-colored (by $I$) directed graph with vertex set $B$ and a
$i$-colored directed edge from $b'$ to $b$ if $\kf_i b' = b$
(equivalently, if $\ke_i b = b'$).  Crystals can be defined in a
more abstract setting where a crystal consists of such a graph along
with maps $\wt : B \to P$ and $\varphi_i, \varepsilon_i : B \to
\Z_{\ge 0}$ satisfying certain axioms.  In this paper, by the
\emph{category of $\g$-crystals} for a symmetrizable Kac-Moody
algebra $\g$, we mean the category consisting of those crystal
graphs $B$ such that each connected component of $B$ is isomorphic
to some $B_\lambda$, the crystal corresponding to the irreducible
highest weight $U_q(\g)$-module of highest weight $\lambda$, where
$\lambda$ is a dominant integral weight.  In this case
\[
    \wt(b) = \mu \text{ for } b \in B^\mu,\quad \varphi_i(b) = \max
    \{k\ |\ \kf_i^kb \ne 0\},\quad \varepsilon_i(b) = \max \{k\ |\ \ke_i^k
    \ne 0\}.
\]
For the rest of this paper, the word \emph{$\g$-crystal} means an
object in this category.

\begin{example}[Crystal bases of finite-dimensional representations
of $U_q(\sl_2)$] \label{eg:B_n}
For $n \in \Z_{\ge 0}$, let $V_n$ be the irreducible
$U_q(\sl_2)$-module of highest weight $n$.  It is a $q$-deformation
of the corresponding $\mathfrak{sl}_2$-module.  Let $v_n$ be a
highest weight vector of $V_n$ and define
\[
    v_{n-2i} = f^{(i)} v_n.
\]
Then $\{v_n, v_{n-2}, \dots, v_{-n}\}$ is a basis of $V_n$ and $v_j$
has weight $j$.  Let
\begin{align*}
    L &= \Span_A \{v_n, v_{n-2}, \dots, v_{-n}\},\quad \text{and}\\
    B &= \{b_n, b_{n-2}, \dots, b_{-n}\}
\end{align*}
where $b_j$ is the image of $v_j$ in the quotient $L/q^{-1/2}L$.  It
is easily checked that $(L,B)$ is a crystal basis of $V_n$ and the
corresponding crystal graph is
\[
    b_n \longrightarrow b_{n-2} \longrightarrow \cdots \longrightarrow b_{-n}
\]
(since there is only one simple root for $\sl_2$, we omit the
edge-coloring).
\end{example}

%%%%%%%%%%%%%%%%%%%%%%%%%%%%%%%%%%%%%%%%%%%%%%%%%%%%%%%%%%%%%%%%%%%%

\subsection{Tensor products}

One of the nicest features of the theory of crystals is the
existence of the \emph{tensor product rule} which tells us how to
form the crystal corresponding to the tensor product of two
representations from the crystals corresponding to the two factors.

\begin{theo}[Tensor product rule {\cite[Theorem~4.4.1]{HK}}, {\cite[Proposition~14.1.14]{CP95}}]
\label{thm:tensor-product-rule}
Suppose $(L_j,B_j)$ are crystal bases of $U_q(\g)$-modules $M_j$
($j=1,2$) in $\mathcal{O}^q_{\text{int}}$.  For $b \in B_j$ and $i
\in I$, let
\begin{gather*}
    \varphi_i(b) = \max\{ k\ |\ \kf_i^k b \ne 0\},\quad
    \varepsilon_i(b) = \max\{ k\ |\ \ke_i^k b \ne 0\}.
\end{gather*}
Let $L = L_1 \otimes_A L_2$ and $B = B_1 \times B_2$. Then
$(L,B)$ is a crystal basis of $M_1 \otimes_{\C_q} M_2$, where the
action of the Kashiwara operators $\ke_i$ and $\kf_i$ are given by
\begin{align*}
    \ke_i (b_1 \otimes b_2) &= \begin{cases} \ke_i b_1 \otimes b_2 &
    \text{if } \varphi_i(b_1) \ge \varepsilon_i(b_2),\\
    b_1 \otimes \ke_i b_2 & \text{if } \varphi_i(b_1) <
    \varepsilon_i(b_2), \end{cases}\\
    \kf_i (b_1 \otimes b_2) &= \begin{cases} \kf_i b_1 \otimes b_2 &
    \text{if } \varphi_i(b_1) > \varepsilon_i(b_2),\\
    b_1 \otimes \kf_i b_2 & \text{if } \varphi_i(b_1) \le
    \varepsilon_i(b_2). \end{cases}
\end{align*}
Here we write $b_1 \otimes b_2$ for $(b_1,b_2) \in B_1 \times B_2$
and $b_1 \times 0 = 0 \times b_2 =0$.
\end{theo}
We write $B_1 \otimes B_2$ for the crystal graph $B_1 \times B_2$ of
$M_1 \otimes M_2$ with crystal operators defined by the formulas in
Theorem~\ref{thm:tensor-product-rule}.  Note that even though we use
a different coproduct than in \cite{HK}, the tensor product rule
remains the same as seen in \cite[Proposition~14.1.14]{CP95}.

%%%%%%%%%%%%%%%%%%%%%%%%%%%%%%%%%%%%%%%%%%%%%%%%%%%%%%%%%%%%%%%%%%%%

\subsection{The braiding in the quantum group}
\label{sec:braiding-quantum-group}

The category of representations of $U(\g)$, the universal enveloping
algebra of a symmetrizable Kac-Moody algebra $\g$, is a symmetric
monoidal category with braiding given by
\[
    \sigma_{U,V} :U \otimes V \to V \otimes U,\quad \sigma(u \otimes v)
    = \flip (u \otimes v) \stackrel{\text{def}}{=}
    v \otimes u \quad \text{for} \quad u \in U,\ v \in V.
\]
However, the analogous map is not a morphism in the category of
representations of $U_q(\g)$ and this category is not a symmetric
monoidal category.  However, it is a braided monoidal category with
a braiding constructed as follows.  The \emph{$R$-matrix} is an
invertible element in a certain completed tensor product $U_q(\g)
\widehat{\otimes} U_q(\g)$ (see \cite{CP95,KT07}).  It defines a map
$U \otimes V \to U \otimes V$ for any representations $U$ and $V$ of
$U_q(\g)$. The map given by
\[
    \sigma_{U,V} : U \otimes V \to V \otimes U,\quad \sigma_{U,V} = \flip
    \circ R
\]
for representations $U$ and $V$ of $U_q(\g)$ is a braiding.

As an example, consider the representation $V_1 \otimes V_1$ of
$U_q(\sl_2)$.  In the basis $S_1 = \{v_1 \otimes v_1, v_{-1} \otimes
v_1, v_1 \otimes v_{-1}, v_{-1} \otimes v_{-1} \}$, the $R$-matrix
is given by (see \cite[Example~6.4.12]{CP95})
\[
    R = q^{-1/2} \begin{pmatrix} q & 0 & 0 & 0 \\ 0 & 1 & 0 & 0 \\ 0
    & q - q^{-1} & 1 & 0 \\ 0 & 0 & 0 & q \end{pmatrix}.
\]
Note that $R|_{q=1} = \id$ and so in the classical limit, the
braiding becomes the map $\flip$.  In the basis $S_1$, we have
\begin{equation} \label{eq:flipR-S1}
    \flip = \begin{pmatrix} 1 & 0 & 0 & 0 \\ 0 & 0 & 1 & 0 \\ 0 & 1
    & 0 & 0 \\ 0 & 0 & 0 & 1 \end{pmatrix}, \quad \text{and so}
    \quad
    \flip \circ R = q^{-1/2} \begin{pmatrix} q&0&0&0 \\ 0& q-q^{-1}
    & 1 & 0 \\ 0&1&0&0 \\ 0&0&0&q \end{pmatrix}.
\end{equation}
Now consider the basis
\[
S_2 = \{v_1 \otimes v_1, a, b, v_{-1} \otimes v_{-1}\},\quad a =
v_{-1} \otimes v_1 - q v_1 \otimes v_{-1},\quad b = v_{-1} \otimes
v_1 + q^{-1}v_1 \otimes v_{-1}.
\]
Note that
\[
    ea = fa = 0,\quad f(v_1 \otimes v_1) = b.
\]
Thus $S_2$ is a basis of $V_1 \otimes V_1$ compatible with the
decomposition $V_1 \otimes V_1 \cong V_0 \oplus V_2$. In the basis
$S_2$, we have
\begin{equation} \label{eq:flipR-S2}
    \flip \circ R = \begin{pmatrix} q^{1/2} & 0 & 0 & 0 \\ 0 &
    -q^{-3/2} & 0 & 0 \\ 0&0& q^{1/2} & 0 \\ 0&0&0& q^{1/2}
    \end{pmatrix}.
\end{equation}
From this form, we easily see that $\flip \circ R$ is an isomorphism
of $U_q(\sl_2)$-modules $V_1 \otimes V_1 \to V_1 \otimes V_1$.  It
acts as multiplication by $q^{1/2}$ on the summand $V_2$ and by
$-q^{-3/2}$ on the summand $V_0$.

Now, from Example~\ref{eg:B_n} and the tensor product rule
(Theorem~\ref{thm:tensor-product-rule}), we see that the crystal
basis of $V_1 \otimes V_1$ is given by
\begin{align*}
    L &= \Span_A \{v_1 \otimes v_1, v_{-1} \otimes v_1, v_1 \otimes v_{-1},
    v_{-1} \otimes v_{-1}\}, \quad \text{and} \\
    B &= \{b_1 \otimes b_1, b_{-1} \otimes b_1, b_1 \times b_{-1},
    b_{-1} \otimes b_{-1}\}.
\end{align*}
From the matrix of $\flip \circ R$ in the basis $S_1$ given in
\eqref{eq:flipR-S1}, we see that it does not preserve the crystal
lattice $L$ since it involves positive powers of $q$. Furthermore,
there is no $\C_q$-multiple of $\flip \circ R$ which preserves $L$
and induces an isomorphism of $L/q^{-1/2}L$.  To see this, note from
\eqref{eq:flipR-S1} that in order for $g(q) \flip \circ R$, with
$g(q) \in \C_q$, to preserve the crystal lattice, we would need
$q^{1/2}g(q) \in A$ and thus $q^{-1/2}g(q) \in q^{-1}A \subseteq
q^{-1/2}A$. However, we would then have
\[
    g(q) \flip \circ R(v_1 \otimes v_{-1}) = q^{-1/2} g(q) v_{-1} \otimes
    v_1 \equiv 0 \mod q^{-1/2} L
\]
and so $g(q) \flip \circ R$ would not induce an isomorphism of
$L/q^{-1/2}L$.  Therefore, we see that the braiding coming from the
$R$-matrix does not pass to the $q \to \infty$ limit.  That is, it
does not induce a braiding on the crystal $B_1 \otimes B_1$.

It turns out that the above phenomenon is unavoidable.  That is, no
braiding on the category of representations of a quantum group
passes to the $q \to \infty$ limit.  In fact, we have the following
even stronger results.

\begin{lem} \label{lem:sl2-not-braided}
The category of $\mathfrak{sl}_2$-crystals cannot be given the
structure of a braided monoidal category.
\end{lem}
\begin{proof}
We prove the result by contradiction.  Suppose the category of
$\mathfrak{sl}_2$-crystals is a braided monoidal category with
braiding $\sigma$. Consider the crystal $B_1$.  It has crystal graph
\[
    B_1\ :\ b_1 \longrightarrow b_{-1}.
\]
The crystal graph of the tensor product $B_1 \otimes B_1$ has two
connected components:
\begin{gather*}
    b_1 \otimes b_1 \longrightarrow b_{-1} \otimes b_1
    \longrightarrow b_{-1} \otimes b_{-1} \quad \cong B_2\\
    b_1 \otimes b_{-1} \quad \cong B_0
\end{gather*}
Since $\sigma_{B_1,B_1}$ is an crystal isomorphism, we see from the
above that it must act as the identity.  Therefore
\begin{equation} \label{eq:sl2-crystal-identity}
    (\id_{B_1} \otimes \sigma_{B_1,B_1}) \circ (\sigma_{B_1,B_1}
    \otimes \id_{B_1}) = \id_{B_1 \otimes B_1 \otimes B_1}.
\end{equation}
Now, the graph of the crystal $B_1 \otimes B_2$ has two connected components:
\begin{gather*}
    b_1 \otimes b_2 \longrightarrow b_{-1} \otimes b_2
    \longrightarrow b_{-1} \otimes b_0 \longrightarrow
    b_{-1} \otimes b_{-2} \quad \cong B_3\\
    b_1 \otimes b_0 \longrightarrow b_1 \otimes b_{-2} \quad \cong B_1
\end{gather*}
The graph of the crystal $B_2 \otimes B_1$ also has two connected components:
\begin{gather*}
    b_2 \otimes b_1 \longrightarrow b_0 \otimes b_1
    \longrightarrow b_{-2} \otimes b_1 \longrightarrow
    b_{-2} \otimes b_{-1} \quad \cong B_3\\
    b_2 \otimes b_{-1} \longrightarrow b_0 \otimes b_{-1} \quad
    \cong B_1
\end{gather*}
Since $\sigma_{B_1,B_2}$ is a crystal isomorphism, we must have
$\sigma_{B_1,B_2}(b_1 \otimes b_0) = b_2 \otimes b_{-1}$.

Now, consider the inclusion of crystals $j: B_2 \hookrightarrow B_1
\otimes B_1$ given by
\[
    j(b_2) = b_1 \otimes b_1,\quad j(b_0) = b_{-1} \otimes b_1,\quad
    j(b_{-2}) = b_{-1} \otimes b_{-1}.
\]
By the naturality of the braiding, the following diagram commutes:
\[
    \xymatrix{ B_1 \otimes B_2 \ar[rr]^{\id_{B_1} \otimes j} \ar[d]^{\sigma_{B_1,B_2}}
    & & B_1 \otimes B_1 \otimes B_1 \ar[d]^{\sigma_{B_1,B_1 \otimes B_1}} \\
    B_2 \otimes B_1 \ar[rr]^{j \otimes \id_{B_1}} & & B_1 \otimes B_1 \otimes B_1}
\]
We therefore have
\begin{multline*}
    \sigma_{B_1,B_1 \otimes B_1} (b_1 \otimes b_{-1} \otimes b_1)
    = \sigma_{B_1,B_1 \otimes B_1} \circ (\id_{B_1} \otimes j)(b_1 \otimes b_0)
    = (j \otimes \id_{B_1}) \circ \sigma_{B_1,B_2} (b_1 \otimes b_0) \\
    = (j \otimes \id_{B_1}) (b_2 \otimes b_{-1}) = b_1 \otimes b_1 \otimes b_{-1}.
\end{multline*}
Comparing to \eqref{eq:sl2-crystal-identity}, we see that
\[
    \sigma_{B_1,B_1 \otimes B_1} \ne (\id_{B_1} \otimes
    \sigma_{B_1,B_1}) \circ (\sigma_{B_1,B_1}
    \otimes \id_{B_1}),
\]
contradicting the fact that $\sigma$ is a braiding.
\end{proof}

An alternative proof of Lemma~\ref{lem:sl2-not-braided} was given in
\cite{HK06}.  We can use the fact that $\g$-crystals, for $\g$ a
symmetrizable Kac-Moody algebra, can be restricted to
$\sl_2$-crystals to generalize this result.

\begin{prop} \label{prop:crystals-not-braided}
For any symmetrizable Kac-Moody algebra $\g$, the category of
$\g$-crystals cannot be given the structure of a braided monoidal
category.
\end{prop}
\begin{proof}
We prove the result by contradiction.  Suppose the category of
$\g$-crystals was a braided monoidal category with braiding $\sigma$
for some symmetrizable Kac-Moody algebra $\g$.  Let $\alpha_1$ and
$\omega_1$ be a simple root and fundamental weight (respectively)
corresponding to some vertex in the Dynkin diagram of $\g$.  The
restriction of a $\g$-crystal to the color 1 yields an
$\sl_2$-crystal.  More precisely, one forgets the operators $\ke_i$,
$\kf_i$, $\varphi_i$ and $\varepsilon_i$ for $i \ne 1$ and projects
the map $\wt$ to the one-dimensional sublattice $\Z \omega_1
\subseteq P$.  In general, even if the original $\g$-crystal was
connected (i.e. irreducible), the induced $\sl_2$-crystal will not
be. However, any morphism of $\g$-crystals induces a morphism of the
restricted $\sl_2$-crystals.

Consider the $\g$-crystal $B_{k\omega_1}$, $k \ge 1$, corresponding
to the irreducible $U_q(\g)$-module of highest weight $k \omega_1$.
If we restrict this to an $\sl_2$-crystal, the connected component
of the $\sl_2$-crystal graph containing the highest weight element
$b_{k \omega_1}$ is isomorphic to the $\sl_2$-crystal $B_k$. Now,
since $b_{\omega_1} \otimes b_{\omega_1}$ is the unique element of
$B_{\omega_1} \otimes B_{\omega_1}$ of weight $2 \omega_1$, we have
\[
    \sigma_{B_{\omega_1}, B_{\omega_1}} (b_{\omega_1} \otimes
    b_{\omega_1}) = b_{\omega_1} \otimes b_{\omega_1}.
\]
The connected $\sl_2$-subcrystal containing the element
$b_{\omega_1} \otimes b_{\omega_1}$ is
\begin{equation} \label{eq:Bomega1-Bomega1-highest-piece}
    b_{\omega_1} \otimes b_{\omega_1} \longrightarrow \kf_1 b_{\omega_1}
    \otimes b_{\omega_1} \longrightarrow \kf_1 b_{\omega_1} \otimes
    \kf_1 b_{\omega_1} \quad \cong B_2
\end{equation}
(as in the proof of Lemma~\ref{lem:sl2-not-braided}).  Thus, since
$\sigma_{B_{\omega_1} \otimes B_{\omega_1}}$ is a morphism of
$\sl_2$-crystals, we must have
\[
    \sigma_{B_{\omega_1} \otimes B_{\omega_1}} (\kf_1 b_{\omega_1}
    \otimes b_{\omega_1}) = \kf_1 b_{\omega_1} \otimes b_{\omega_1}.
\]
Now, the only element of $B_{\omega_1} \otimes B_{\omega_1}$ of
weight $2 \omega_1 - \alpha_1$ not contained in the connected
$\sl_2$-subcrystal mentioned above is $b_{\omega_1} \otimes \kf_1
b_{\omega_1}$. Therefore, we must also have
\[
    \sigma_{B_{\omega_1}, B_{\omega_1}} (b_{\omega_1} \otimes
    \kf_1 b_{\omega_1}) = b_{\omega_1} \otimes \kf_1 b_{\omega_1}.
\]
Thus,
\begin{multline} \label{eq:g-crystal-identity}
    (\id_{B_{\omega_1}} \otimes \sigma_{B_{\omega_1},
    B_{\omega_1}}) \circ (\sigma_{B_{\omega_1},B_{\omega_1}}
    \otimes \id_{B_{\omega_1}})
    (b_{\omega_1} \otimes \kf_1 b_{\omega_1} \otimes
    b_{\omega_1}) \\
    = (\id_{B_{\omega_1}} \otimes \sigma_{B_{\omega_1},
    B_{\omega_1}}) (b_{\omega_1} \otimes \kf_1 b_{\omega_1}
    \otimes b_{\omega_1}) = b_{\omega_1} \otimes \kf_1 b_{\omega_1}
    \otimes b_{\omega_1}.
\end{multline}

Now, the connected $\sl_2$-subcrystal of $B_{\omega_1} \otimes B_{2
\omega_1}$ containing the element $b_{\omega_1} \otimes b_{2
\omega_1}$ is
\[
    b_{\omega_1} \otimes b_{2 \omega_1} \longrightarrow \kf_1
    b_{\omega_1} \otimes b_{2 \omega_1} \longrightarrow \kf_1
    b_{\omega_1} \otimes \kf_1 b_{2 \omega_1} \longrightarrow \kf_1
    b_{\omega_1} \otimes \kf_1^2 b_{2 \omega_1} \quad \cong B_3,
\]
and the connected $\sl_2$-subcrystal containing the element
$b_{\omega_1} \otimes \kf_1 b_{2 \omega_1}$ is
\[
    b_{\omega_1} \otimes \kf_1 b_{2 \omega_1} \longrightarrow
    b_{\omega_1} \otimes \kf_1^2 b_{2 \omega_1} \quad \cong B_1
\]
(as in the proof of Lemma~\ref{lem:sl2-not-braided}).  Similarly, we
have the following $\sl_2$-subcrystals of $B_{2 \omega_1} \otimes
B_{\omega_1}$:
\begin{gather*}
    b_{2 \omega_1} \otimes b_{\omega_1} \longrightarrow \kf_1 b_{2
    \omega_1} \otimes b_{\omega_1} \longrightarrow \kf_1^2 b_{2
    \omega_1} \otimes b_{\omega_1} \longrightarrow \kf_1^2 b_{2
    \omega_1} \otimes \kf_1 b_{\omega_1} \quad \cong B_3, \\
    b_{2 \omega_1} \otimes \kf_1 b_{\omega_1} \longrightarrow \kf_1
    b_{2 \omega_1} \otimes \kf_1 b_{\omega_1} \quad \cong B_1.
\end{gather*}
Now, since $b_{\omega_1} \otimes b_{2 \omega_1}$ and $b_{2 \omega_1}
\otimes b_{\omega_1}$ are the unique elements of $B_{\omega_1}
\otimes B_{2 \omega_1}$ and $B_{2 \omega_1} \otimes B_{\omega_1}$
(respectively) of weight $3 \omega_1$, we must have
\[
    \sigma_{B_{\omega_1}, B_{2 \omega_1}} (b_{\omega_1} \otimes b_{2
    \omega_1}) = b_{2 \omega_1} \otimes b_{\omega_1}.
\]
Also, since $b_{\omega_1} \otimes \kf_1 b_{2 \omega_1}$ and $b_{2
\omega_1} \otimes \kf_1 b_{\omega_1}$ are the only elements of
$B_{\omega_1} \otimes B_{2 \omega_1}$ and $B_{2 \omega_1} \otimes
B_{\omega_1}$ (respectively) of weight $3 \omega_1 - \alpha_1$ not
contained in the connected $\sl_2$-subcrystal containing
$b_{\omega_1} \otimes b_{2 \omega_1}$ and $b_{2 \omega_1} \otimes
b_{\omega_1}$ (respectively), we must have
\[
    \sigma_{B_{\omega_1}, B_{2 \omega_1}} (b_{\omega_1} \otimes \kf_1 b_{2
    \omega_1}) = b_{2 \omega_1} \otimes \kf_1 b_{\omega_1}.
\]

Now, consider the inclusion of $\g$-crystals $j : B_{2 \omega_1}
\hookrightarrow B_{\omega_1} \otimes B_{\omega_1}$ determined by $j
(b_{2 \omega_1}) = b_{\omega_1} \otimes b_{\omega_1}$.  Restricting
to the connected $\sl_2$-crystals containing the elements $b_{2
\omega_1}$ and $b_{\omega_1} \otimes b_{\omega_1}$, we see from
\eqref{eq:Bomega1-Bomega1-highest-piece} that
\[
    j(\kf_1 b_{2 \omega_1}) = \kf_1 b_{\omega_1} \otimes
    b_{\omega_1},\quad j(\kf_1^2 b_{2 \omega_1}) = \kf_1
    b_{\omega_1} \otimes \kf_1 b_{\omega_1}.
\]

By the naturality of the braiding $\sigma$, the following diagram is
commutative:
\[
    \xymatrix{ B_{\omega_1} \otimes B_{2 \omega_1} \ar[rr]^{\id_{B_{\omega_1}}
    \otimes j} \ar[d]^{\sigma_{B_{\omega_1},B_{2 \omega_1}}}
    & & B_{\omega_1} \otimes B_{\omega_1} \otimes B_{\omega_1}
    \ar[d]^{\sigma_{B_{\omega_1},B_{\omega_1} \otimes B_{\omega_1}}} \\
    B_{2 \omega_1} \otimes B_{\omega_1} \ar[rr]^{j \otimes
    \id_{B_{\omega_1}}} & & B_{\omega_1} \otimes B_{\omega_1} \otimes
    B_{\omega_1}}
\]
Therefore
\begin{align*}
    \sigma_{B_{\omega_1}, B_{\omega_1} \otimes B_{\omega_1}} (b_{\omega_1}
    \otimes \kf_1 b_{\omega_1} \otimes b_{\omega_1}) &=
    \sigma_{B_{\omega_1}, B_{\omega_1} \otimes B_{\omega_1}} \circ
    (\id_{B_{\omega_1}} \otimes j) (b_{\omega_1} \otimes \kf_1 b_{2
    \omega_1}) \\
    &= (j \otimes \id_{B_{\omega_1}}) \circ \sigma_{B_{\omega_1},
    B_{2 \omega_1}} (b_{\omega_1} \otimes \kf_1 b_{2 \omega_1}) \\
    &= (j \otimes \id_{B_{\omega_1}}) (b_{2 \omega_1} \otimes \kf_1
    b_{\omega_1}) \\
    &= b_{\omega_1} \otimes b_{\omega_1} \otimes \kf_1 b_{\omega_1}.
\end{align*}
Comparing this to \eqref{eq:g-crystal-identity} we see that
\[
    \sigma_{B_{\omega_1}, B_{\omega_1} \otimes
    B_{\omega_1}} \ne  (\id_{B_{\omega_1}} \otimes \sigma_{B_{\omega_1},
    B_{\omega_1}}) \circ (\sigma_{B_{\omega_1},B_{\omega_1}}
    \otimes \id_{B_{\omega_1}}).
\]
This contradicts the fact that $\sigma$ is a braiding.
\end{proof}

%%%%%%%%%%%%%%%%%%%%%%%%%%%%%%%%%%%%%%%%%%%%%%%%%%%%%%%%%%%%%%%%%%%%%%%

\section{Crystals and coboundary categories}
\label{sec:crystal-commutor}

In this section, we discuss how the categories of $U_q(\g)$-modules
and $\g$-crystals can be given the structure of a coboundary
category.  For the case of $\g$-crystals, we mention several
different constructions and note the relationship between them.

%%%%%%%%%%%%%%%%%%%%%%%%%%%%%%%%%%%%%%%%%%%%%%%%%%%%%%%%%%%%%%%%%%%%%%%

\subsection{Drinfel\cprime d's unitarization}
\label{sec:Drinfeld-unitarization}

In \cite{Dri90}, Drinfel\cprime d defined the \emph{unitarized
$R$-matrix}
\[
    \bar R = R(R^{\text{op}} R)^{-1/2},
\]
where $R^{\text{op}} = \flip (R)$ (as an operator on $M_1 \otimes
M_2$, $R^{\text{op}}$ acts as $\flip \circ R \circ \flip$) and the
square root is taken with respect to a certain filtration on the
completed tensor product $U_q(\g) \hat \otimes U_q(\g)$.  He then
showed that $\flip \circ \bar R$ is a cactus commutor and so endows
the category of $U_q(\g)$-modules with the structure of a coboundary
category (see the comment after the proof of Proposition~3.3 in
\cite{Dri90}). That is, it satisfies the conditions of
Definition~\ref{def:coboundary-cat}.

Consider the representation $V_1 \otimes V_1$ of
$U_q(\mathfrak{sl}_2)$. In Section~\ref{sec:braiding-quantum-group},
we described the action of the $R$-matrix on this representation in
two different bases $S_1$ and $S_2$. It follows from
\eqref{eq:flipR-S2} that in the basis $S_2$,
\[
    R^{\text{op}} R = (\flip \circ R)^2 = \begin{pmatrix} q & 0 & 0 & 0 \\ 0 & q^{-3} & 0 & 0
    \\ 0 & 0 & q & 0 \\ 0 & 0 & 0 & q \end{pmatrix}.
\]
Therefore, we can take the (inverse of the) square root
\[
    (R^{\text{op}}R)^{-1/2} = q^{-1/2} \begin{pmatrix} 1 & 0 & 0 & 0 \\ 0 &
    q^2 & 0 & 0 \\ 0 & 0 & 1 & 0 \\ 0 & 0 & 0 & 1
    \end{pmatrix}.
\]
In the basis $S_1$, we then have
\begin{gather*}
     (R^{\text{op}}R)^{-1/2} = q^{-1/2} \begin{pmatrix} 1&0&0&0 \\ 0 &
     \frac{2q^2}{1+q^2} & \frac{q-q^3}{1+q^2} & 0 \\ 0 &
     \frac{q-q^3}{1+q^2} & \frac{1+q^4}{1+q^2} & 0 \\ 0 & 0 & 0 & 1
     \end{pmatrix},\\
     \bar R = R
     (R^{\text{op}}R)^{-1/2} =
     \begin{pmatrix} 1 & 0 & 0 & 0 \\ 0 & \frac{2q}{1+q^2} &
     \frac{1-q^2}{1+q^2} & 0 \\ 0 & \frac{q^2-1}{1+q^2} &
     \frac{2q}{1+q^2} & 0 \\ 0&0&0&1 \end{pmatrix}.
\end{gather*}
Therefore, in the basis $S_1$,
\begin{equation} \label{eq:commutors-S1}
    \flip \circ R = q^{-1/2} \begin{pmatrix} q&0&0&0 \\ 0& q-q^{-1}
    & 1 & 0 \\ 0&1&0&0 \\ 0&0&0&q \end{pmatrix},\quad \flip \circ
    \bar R = \begin{pmatrix} 1 & 0 & 0 & 0 \\ 0 & \frac{q^2-1}{1+q^2} &
     \frac{2q}{1+q^2} & 0 \\ 0 & \frac{2q}{1+q^2} &
     \frac{1-q^2}{1+q^2} & 0 \\ 0&0&0&1 \end{pmatrix}.
\end{equation}
And in the basis $S_2$,
\begin{equation} \label{eq:commutors-S2}
    \flip \circ R = \begin{pmatrix} q^{1/2} & 0 & 0 & 0 \\ 0 &
    -q^{-3/2} & 0 & 0 \\ 0&0& q^{1/2} & 0 \\ 0&0&0& q^{1/2}
    \end{pmatrix},\quad \flip \circ \bar R = \begin{pmatrix} 1&0&0&0
    \\ 0&-1&0&0 \\ 0&0&1&0 \\ 0&0&0&1 \end{pmatrix}.
\end{equation}
In the above, we have recalled the computation of $\flip \circ R$
from Section~\ref{sec:braiding-quantum-group} for the purposes of
comparison.

We note two important properties of $\flip \circ \bar R$. First of
all, we see from \eqref{eq:commutors-S1} that the matrix
coefficients in the basis $S_1$ of $\flip \circ \bar R$ lie in $A$,
the ring of rational functions in $\C_q$ that are regular at $q =
\infty$. Thus, $\flip \circ \bar R$ preserves the crystal lattice of
$V_1 \otimes V_1$.  In the $q \to \infty$ limit (more precisely,
when passing to the quotient $L/q^{-1/2} L$), we have (in the basis
$S_1$)
\[
    \flip \circ \bar R
    = \begin{pmatrix} 1&0&0&0 \\ 0&1&0&0 \\
    0&0&-1&0 \\ 0&0&0&1 \end{pmatrix} \mod q^{-1/2} L.
\]
Thus, $\flip \circ \bar R$ passes to the $q \to \infty$ limit and,
up to signs, induces an involution on the crystal $B_1 \otimes B_1$
(see Section~\ref{sec:various-commutors} for details). As noted in
Section~\ref{sec:braiding-quantum-group}, the same is not true of
$\flip \circ R$.

The second important property of $\flip \circ \bar R$ is that it is
a cactus commutor.  We see immediately from \eqref{eq:commutors-S2}
that $(\flip \circ \bar R)^2=\id$.  A straightforward (if somewhat
lengthy) computation shows that $\flip \circ \bar R$ also satisfies
the cactus relation (see \cite[\S3]{Dri90} for the proof in a more
general setting).

The unitarized $R$-matrix has shown up in several different places.
In \cite{BZ07} it arose naturally in the development of the theory
of braided symmetric and exterior algebras.  The reason for this is
that if one wants to have interesting symmetric or exterior
algebras, one needs an operator with eigenvalues of positive or
negative one.  Notice from \eqref{eq:commutors-S2} that while $\flip
\circ R$ does not have this property, the operator $\flip \circ \bar
R$ does. Essentially, in the basis $S_2$, the matrix for $\flip
\circ \bar R$ is obtained from the matrix for $\flip \circ R$ by
setting $q=1$.  Note that this does not imply that $\flip \circ \bar
R$ is an operator in the classical limit, merely that its matrix
coefficients in a certain basis do not involve powers of $q$.  The
$q \to \infty$ limit of the unitarized $R$-matrix also appeared
independently in the study of cactus commutors for crystals.  We
discuss this in the next two subsections.

%%%%%%%%%%%%%%%%%%%%%%%%%%%%%%%%%%%%%%%%%%%%%%%%%%%%%%%%%%%%%%%%%%%%%%%

\subsection{The crystal commutor using the Sch{\" u}tzenberger involution}

Let $\g$ be a simple complex Lie algebra and let $I$ denote the set
of vertices of the Dynkin graph of $\g$. If $w_0$ is the long
element in the Weyl group of $\g$, let $\theta : I \to I$ be the
involution such that $\alpha_{\theta(i)} = - w_0 \cdot \alpha_i$.
Define a crystal $\overline{B_\lambda}$ with underlying set
$\{\overline{b}\ |\ b \in B_\lambda\}$ and
\[
    \ke_i \cdot \overline{b} = \overline{\kf_{\theta(i)} \cdot b},\quad
    \kf_i \cdot \overline{b} = \overline{\ke_{\theta(i)} \cdot b},\quad
    \wt(\overline{b}) = w_0 \cdot \wt(b).
\]
There is a crystal isomorphism $\overline{B_\lambda} \cong
B_\lambda$.  We compose this isomorphism with the map of sets
$B_\lambda \to \overline{B_\lambda}$ given by $b \mapsto
\overline{b}$ and denote the resulting map by $\xi = \xi_{B_\lambda}
: B_\lambda \to B_\lambda$.  We call the map $\xi$ the \emph{Sch{\"
u}tzenberger involution}.  When $\g = \mathfrak{gl}_n$, there is a
realization of $B_\lambda$ using tableaux.  In this realization,
$\xi$ is the usual Sch{\" u}tzenberger involution on tableaux (see
\cite{LLT95}).

For an arbitrary $\g$-crystal $B$, write $B = \bigoplus_{i=1}^k
B_{\lambda_i}$.  This is a decomposition of $B$ into connected
components.  Then define $\xi_B : B \to B$ by $\xi_B =
\bigoplus_{i=1}^k \xi_{B_{\lambda_i}}$.  That is, we apply
$\xi_{B_{\lambda_i}}$ to each connected component $B_{\lambda_i}$.

For crystals $A$ and $B$, define
\[
    \sigma^S : A \otimes B \to B \otimes A,\quad \sigma^S(a \otimes
    b) = \xi_{B \otimes A} (\xi_B(b) \otimes \xi_A(a)).
\]

\begin{theo}[{\cite[{Proposition~3, Theorem~3}]{HK06}}]
We have
\begin{enumerate}
    \item $\sigma^S_{B,A} \circ \sigma^S_{A,B} = \id$, and
    \item $\sigma^S$ satisfies the cactus relation
        \eqref{eq:cactus-relation}.
\end{enumerate}
In other words, $\sigma^S$ endows the category of $\g$-crystals with
the structure of a coboundary category.
\end{theo}

%%%%%%%%%%%%%%%%%%%%%%%%%%%%%%%%%%%%%%%%%%%%%%%%%%%%%%%%%%%%%%%%%%%%%%%

\subsection{The crystal commutor using the Kashiwara involution}
\label{sec:commutor-kash-involution}

Let $\g$ be a symmetrizable Kac-Moody algebra and let $B_\infty$ be
the $\g$-crystal corresponding to the lower half $U_q^-(\g)$ of the
associated quantized universal enveloping algebra.  Let $* : U_q(\g)
\to U_q(\g)$ be the $\C_q$-linear anti-automorphism given by
\[
\quad e_i^* = e_i, \quad f_i^* = f_i, \quad \left(q^h\right)^* =
q^{-h}.
\]
The map $*$ sends $U_q^-(\g)$ to $U_q^-(\g)$ and induces a map $* :
B_\infty \to B_\infty$ (see \cite[\S 8.3]{Kas95}).  We call the map
$*$ the \emph{Kashiwara involution}.

Let $B_\lambda$ be the $\g$-crystal corresponding to the irreducible
highest weight $U_q(\g)$-module of highest weight $\lambda$ and let
$b_\lambda$ be its highest weight element.  For two integral
dominant weights $\lambda$ and $\mu$, there is an inclusion of
crystals $B_{\lambda + \mu} \hookrightarrow B_\lambda \otimes B_\mu$
sending $b_{\lambda + \mu}$ to $b_\lambda \otimes b_\mu$.  It
follows from the tensor product rule that the image of this
inclusion contains all elements of the form $b \otimes b_\mu$ for $b
\in B_\lambda$. Thus we define a map
\[
\iota^{\lambda+\mu}_\lambda : B_\lambda \to B_{\lambda + \mu}
\]
which sends $b \in B_\lambda$ to the inverse image of $b \otimes
b_\mu$ under the inclusion $B_{\lambda + \mu} \hookrightarrow
B_\lambda \otimes B_\mu$.  While this map is not a morphism of
crystals, it is $\ke_i$-equivariant for all $i$ and takes
$b_\lambda$ to $b_{\lambda + \mu}$.

The maps $\iota^{\lambda + \mu}_\lambda$ make the family of crystals
$B_\lambda$ into a directed system and the crystal $B_\infty$ can be
viewed as the limit of this system.  We have $\ke_i$-equivariant
maps $\iota^\infty_\lambda : B_\lambda \to B_\infty$ which we will
simply denote by $\iota^\infty$ when it will cause no confusion.
Define $\varepsilon^* : B_\infty \to P_+$ by
\[
\varepsilon^*(b) = \min \{\lambda\ |\ b \in
\iota^\infty(B_\lambda)\}
\]
where we put the usual order on $P_+$, the positive weight lattice
of $\g$, given by $\lambda \ge \mu$ if and only if $\lambda - \mu
\in Q_+$. Recall that we also have the map $\varepsilon : B_\infty
\to P_+$ given by $\varepsilon(b)(h_i) = \varepsilon_i(b)$. Then by
\cite[Proposition~8.2]{Kas95}, the Kashiwara involution preserves
weights and satisfies
\begin{equation}
\varepsilon^*(b) = \varepsilon(b^*).
\end{equation}

Consider the crystal $B_\lambda \otimes B_\mu$.  Since $\varphi(b) =
\varepsilon(b) + \wt(b)$ for all $b \in B_\lambda$, we have that
$\varphi(b_\lambda) = \wt(b_\lambda) = \lambda$. It follows from the
tensor product rule for crystals that the highest weight elements of
$B_\lambda \otimes B_\mu$ are those elements of the form $b_\lambda
\otimes b$ for $b \in B_\mu$ with $\varepsilon(b) \le \lambda$. Thus
$\varepsilon^*(b^*) = \varepsilon(b) \le \lambda$ and so, by the
definition of $\varepsilon^*$, we have $b^* \in
\iota^\infty(B_\lambda)$. So we can consider $b^*$ as an element of
$B_\lambda$.  Furthermore, $\varepsilon(b^*) = \varepsilon^*(b) \le
\mu = \varphi(b_\mu)$ since $b \in B_\mu$. Thus $b_\mu \otimes b^*$
is a highest weight element of $B_\mu \otimes B_\lambda$.  Since
$B_\lambda \otimes B_\mu \cong B_\mu \otimes B_\lambda$ as crystals,
we can make the following definition.

\begin{defin}[{\cite[\S3]{KT06}}]
Let $\sigma^c_{B_\lambda,B_\mu} : B_\lambda \otimes B_\mu
\stackrel{\cong}{\to} B_\mu \otimes B_\lambda$ be the crystal
isomorphism given uniquely by $\sigma^c_{B_\lambda, B_\mu}(b_\lambda
\otimes b) = b_\mu \otimes b^*$ for $b_\lambda \otimes b$ a highest
weight element of $B_\lambda \otimes B_\mu$.
\end{defin}

\begin{theo}[{\cite[Theorem~3.1]{KT06}}] \label{thm:schutzenberger=kashiwara}
For $\g$ a simple complex Lie algebra, $\sigma^S = \sigma^c$ and so
$\sigma^c$ satisfies the cactus relation.
\end{theo}

We call $\sigma^c$ the \emph{crystal commutor}.  Note that
Theorem~\ref{thm:schutzenberger=kashiwara} only implies that it is a
cactus commutor for $\g$ of finite type.

%%%%%%%%%%%%%%%%%%%%%%%%%%%%%%%%%%%%%%%%%%%%%%%%%%%%%%%%%%%%%%%%%%%%%%%

\subsection{The relationship between the various commutors}
\label{sec:various-commutors}

We have described three ways of constructing commutors in the
categories of $U_q(\g)$-modules or $\g$-crystals.  The three
definitions are closely related.  In \cite{HK06}, Henriques and
Kamnitzer defined a cactus commutor on the category of
finite-dimensional $U_q(\g)$-modules when $\g$ is of finite type,
using an analogue of the Sch{\" u}tzenberger involution on
$U_q(\g)$. This definition of the commutor involves some choices of
normalization. In \cite{KT07}, Kamnitzer and Tingley showed that
Drinfel\cprime d's commutor coming from the unitarized $R$-matrix
corresponds to the commutor coming from the Sch{\" u}tzenberger
involution up to normalization.  From there it follows that
Drinfel\cprime d's commutor preserves crystal lattices and acts on
crystal bases as the crystal commutor, up to signs.  The precise
statement is the following.

\begin{prop}[{\cite[Theorem~9.2]{KT07}}]
Suppose $(L_j,B_j)$ are crystal bases of two finite-dimensional
representations $V_j$, $j=1,2$, of $U_q(\g)$ for a simple complex
Lie algebra $\g$. Let $\sigma^D_{V_1,V_2}$ be the isomorphism $V_1
\otimes V_2 \cong V_2 \otimes V_1$ given by $\flip \circ \bar R$.
Then
\[
    \sigma^D_{V_1,V_2}(L_1 \otimes L_2) = L_2 \otimes L_1
\]
and thus $\sigma^D_{L_1 \otimes L_2}$ induces a map
\[
    \sigma^{D \mod q^{-1/2}(L_1 \otimes L_2)}_{V_1 \otimes V_2} : (L_1
    \otimes L_2)/q^{-1/2} (L_1 \otimes L_2) \to (L_2 \otimes
    L_1)/q^{-1/2} (L_2 \otimes L_1).
\]
For all $b_j \in B_j$, $j=1,2$,
\[
    \sigma^{D \mod q^{-1/2}(L_1 \otimes L_2)}_{V_1 \otimes V_2} (b_1
    \otimes b_2) = (-1)^{\left<\lambda + \mu - \nu,
    \rho^\vee\right>} \sigma^c_{B_1,B_2}(b_1 \otimes b_2)
\]
where $\lambda$, $\mu$ and $\nu$ are the highest weights of the
connected components of $B_1$, $B_2$ and $B_1 \otimes B_2$
containing $b_1$, $b_2$ and $b_1 \otimes b_2$ respectively, $\rho$
is half the sum of the positive roots of $\g$ and $\left< \cdot,
\cdot \right>$ denotes the pairing between the Cartan subalgebra
$\mathfrak{h} \subset \g$ and its dual $\mathfrak{h}^*$.
\end{prop}

We thus have essentially two definitions of the crystal commutor.
The first, using the Sch{\" u}tzenberger involution (and coinciding
with Drinfel\cprime d's commutor in the crystal limit) only applies
to $\g$ of finite type but with this definition, it is apparent that
the commutor satisfies the cactus relation.  The second definition,
using the Kashiwara involution, applies to $\g$ of arbitrary type
but it is not easy to see that it satisfies the cactus relation. In
the next section, we will explain how a geometric interpretation of
this commutor using quiver varieties allows one to prove that this
is indeed the case.

%%%%%%%%%%%%%%%%%%%%%%%%%%%%%%%%%%%%%%%%%%%%%%%%%%%%%%%%%%%%%%%%%%%%%%%%

\section{A geometric realization of the crystal commutor}
\label{sec:geom-commutor}

In this section we describe a geometric realization of the crystal
commutor defined in Section~\ref{sec:commutor-kash-involution} in
the language of quiver varieties.  This realization yields new
insight into the coboundary structure and equips us with new
geometric tools.  Using these tools, one is able to show that the
category of $\g$-crystals for an arbitrary symmetrizable Kac-Moody
algebra $\g$ can be given the structure of a coboundary category.
This extends the previously known result, which held for $\g$ of
finite type.

%%%%%%%%%%%%%%%%%%%%%%%%%%%%%%%%%%%%%%%%%%%%%%%%%%%%%%%%%%%%%%%%%%%%%%%%

\subsection{Quiver varieties}

Lusztig \cite{L91}, Nakajima \cite{Nak94, Nak98, Nak01} and Malkin
\cite{Mal03} have introduced varieties associated to quivers
(directed graphs) built from the Dynkin graph of a Kac-Moody algebra
$\g$ with symmetric Cartan matrix.  These varieties yield geometric
realizations of the quantum group $U_q(\g)$, the representations of
$\g$, and tensor products of these representations in the homology
(or category of perverse sheaves) of such varieties.  In addition,
Kashiwara and Saito \cite{KS97,S02}, Nakajima \cite{Nak01} and
Malkin \cite{Mal03} have used quiver varieties to give a geometric
realization of the crystals of these objects. Namely, they defined
geometric operators on the sets of irreducible components of quiver
varieties, endowing these sets with the structure of crystals.  In
the current paper, we will focus on the $\mathfrak{sl}_2$ case of
these varieties for simplicity.  In this case, the quiver varieties
are closely related to grassmannians and flag varieties.

In this section, all vector spaces will be complex. Fix integers $w
\ge 0$ and $n \ge 1$, and $\w = (w_i)_{i=1}^n \in (\Z_{\ge 0})^n$
such that $\sum_{i=1}^n w_i = w$.  Let $W$ be a $w$-dimensional
vector space and let
\[
    0 = W_0 \subseteq W_1 \subseteq \dots \subseteq W_n = W,\quad
    \dim W_i/W_{i-1} = w_i \text{ for } 1 \le i \le n,
\]
be an $n$-step partial flag in $W$.  Define the \emph{tensor product
quiver variety}
\[
    \T(\w) = \{ (U,t)\ |\ U \subseteq W,\ t \in \End W,\ t(W_i)
    \subseteq W_{i-1}\ \forall\, i,\ \im t \subseteq U \subseteq \ker t\}.
\]
We use the notation $\T(\w)$ since, up to isomorphism, this variety
depends only on the dimensions of the subspaces $W_i$, $0 \le i \le
n$. We have
\[
    \T(\w) = \bigsqcup_{u=0}^w \T(u,\w),\quad \text{where}\quad
    \T(u,\w) = \{(U,t) \in \T(\w)\ |\ \dim U = u\}.
\]
Let $B(u,\w)$ denote the set of irreducible components of $\T(u,\w)$
and set $B(\w) = \bigsqcup_u B(u,\w)$.

\begin{comment}
For $0 \le u_1 \le u_2 \le w$, let
\[
    \T(u_1,u_2,\w) = \{(U_1,U_2,t)\ |\ (U_1,t) \in \T(u_1,\w),\ (U_2,t) \in
    \T(u_2,\w),\ U_1 \subseteq U_2\}.
\]
We then have natural projections
\begin{equation} \label{eq:TPV-two-projections}
    \T(u_1,\w) \stackrel{p_1}{\longleftarrow} \T(u_1,u_2,\w)
    \stackrel{p_2}{\longrightarrow} \T(u_2,\w)
\end{equation}
where $p_1(U_1,U_2,t) = (U_1,t)$, $p_2(U_1,U_2,t) = (U_2,t)$.
\end{comment}

Define
\begin{gather*}
    \wt : B(\w) \to P,\quad \wt (X) = w-2u \text{ for } X \in
    B(u,\w),\\
    \varepsilon : \T(\w) \to \Z_{\ge 0},\quad \varepsilon(U,t) =
    \dim U/\im t,\\
    \varphi : \T(\w) \to \Z_{\ge 0}, \quad \varphi(U,t) = \dim \ker t/U
    = \varepsilon(U,t) + w -2\dim U.
\end{gather*}
For $k \in \Z_{\ge 0}$, let
\[
    \T(u,\w)_k = \{(U,t) \in \T(u,\w)\ |\ \varepsilon(U,t) = k\},
\]
and for $X \in B(u,\w)$, define $\varepsilon(X)=\varepsilon(U,t)$
and $\varphi(X) = \varepsilon(U,t)$ for a generic point $(U,t)$ of
$X$. Let
\[
    B(u,\w)_k = \{X \in B(u,\w)\ |\ \varepsilon(X)=k\},\quad B(\w)_k
    = \bigsqcup_u B(u,\w)_k.
\]
The map
\[
    \T(u,\w)_k \to \T(u-k,\w)_0,\quad (U,t) \mapsto (\im t, t)
\]
is a Grassmann bundle and thus induces an isomorphism
\begin{equation} \label{eq:irrcomp-isom}
    B(u,\w)_k \cong B(u-k,\w)_0.
\end{equation}
We then define crystal operators on $B(\w)$ as follows. Suppose $X'
\in B(u-k,\w)_0$ corresponds to $X \in B(u,\w)_k$ under the
isomorphism \eqref{eq:irrcomp-isom}. Define
\begin{align*}
    \kf^k : B(u-k,\w)_0 \to B(u,\w)_k,\quad \kf^k(X')=X,\\
    \ke^k : B(u,\w)_k \to B(u-k,\w)_0,\quad \ke^k(X)=X'.
\end{align*}
For $k > 0$, we then define $\ke_i : B(\w) \to B(\w)$ by
\[
    \ke : B(u,\w)_k \stackrel{\ke^k}{\longrightarrow} B(u-k,\w)_0
    \stackrel{\kf^{k-1}}{\longrightarrow} B(u-1;\w)_{k-1},
\]
and set $\ke_i(X) = 0$ for $X \in B(u,\w)_0$.  For $k > 2u - w$,
define
\[
    \kf : B(u,\w)_k \stackrel{\ke^k}{\longrightarrow} B(u-k,\w)_0
    \stackrel{\kf^{k+1}}{\longrightarrow} B(u+1,\w)_{k+1},
\]
and set $\kf(X)=0$ for $X \in B(u,\w)_k$ with $k \le 2u-w$.  The
maps $\ke^k$ and $\kf^k$ defined above can be considered as the
$k$th powers of $\ke$ and $\kf$ respectively.

\begin{theo}[{\cite[\S7]{Nak01}}] \label{thm:tpqv}
The operators $\varepsilon, \varphi, \wt, \ke$, and $\kf$ endow the
set $B(\w)$ with the structure of an $\sl_2$-crystal and $B(\w)
\cong B_{w_1} \otimes \dots \otimes B_{w_n}$ as $\sl_2$-crystals.
\end{theo}

We let $\phi : B(\w) \cong B_{w_1} \otimes \dots \otimes B_{w_n}$
denote the isomorphism of Theorem~\ref{thm:tpqv}.

%%%%%%%%%%%%%%%%%%%%%%%%%%%%%%%%%%%%%%%%%%%%%%%%%%%%%%%%%%%%%%%%%%%%%

\subsection{The geometric realization of the crystal commutor}

Fix a hermitian form on $W$.  Let $t^\dag$ denote the hermitian
adjoint of $t \in \End W$ and let $S^\bot$ denote the orthogonal
complement to a subspace $S \subseteq W$. If we let $\hat W_i =
W_{n-i}^\bot$ for $0 \le i \le n$, and $\hat \w = (\hat
w_i)_{i=1}^n$ where
\[
\hat w_i = \dim \hat W_i/ \hat W_{i-1} = \dim W_{n-i}^\bot /
W_{n-i+1}^\bot = w_{n-i+1},
\]
then
\[
    \T(\hat \w) = \{ (U, t)\ |\ U \subseteq W,\ t \in \End W,\
    t(\hat W_i) \subseteq \hat W_{i-1}\ \forall\, i,\ \im t \subseteq U
    \subseteq \ker t\}.
\]
Note that $\varepsilon(U,t)=0$ if and only if $U = \im t$.  Also,
for $t \in \End W$,
\begin{gather*}
    t(W_i) \subseteq W_{i-1} \Rightarrow t^\dag(\hat W_{n-i+1}) \subseteq
    \hat W_{n-i}.
\end{gather*}
Therefore,
\[
    (\im t, t) \in \T(\w) \iff (\im t^\dag, t^\dag) \in \T(\hat \w),
\]
and the map $(\im t, t) \mapsto (\im t^\dag, t^\dag)$ induces
isomorphisms
\[
    \T(u,\w)_0 \cong \T(u,\hat \w)_0,\quad B(u,\w)_0 \cong B(u,\hat \w)_0.
\]
We denote the isomorphism $B(u,\w)_0 \cong B(u,\hat \w)_0$ by $X
\mapsto X^\dag$ for $X \in B(u,\w)_0$.  Since the elements of
$B(\w)_0$ are precisely the highest weight elements of the crystal
$B(\w)$, a commutor is uniquely determined by its action on these
elements.

\begin{theo}[{\cite[\S4.2]{Sav08a}}]
\begin{enumerate}
    \item If $n=2$ and $X \in B(\w)_0$, we have
        \[
            \phi^{-1} \circ \sigma^c_{B_{w_1}, B_{w_2}} \circ \phi (X) = X^\dag,
        \]
        and thus the map $X \mapsto X^\dag$ corresponds
        to the crystal commutor on highest weight elements.

    \item If $n=3$ and $X \in B(\w)_0$,
        \begin{multline*}
            \phi^{-1} \circ \left(\sigma^c_{B_{w_1}, B_{w_3} \otimes B_{w_2}}
            \circ \left(\id_{B_{w_1}} \otimes \sigma^c_{B_{w_2} \otimes
            B_{w_3}}\right)\right) \circ \phi(X) = X^\dag \\
            = \phi^{-1} \circ
            \left( \sigma^c_{B_{w_2} \otimes B_{w_1}, B_{w_3}} \circ \left(
            \sigma^c_{B_{w_1}, B_{w_2}} \otimes \id_{B_{w_3}} \right) \right)
            \circ \phi (X),
        \end{multline*}
        and thus the crystal commutor satisfies the cactus relation.
\end{enumerate}
\end{theo}

One advantage of the geometric interpretation of the crystal
commutor defined here is that it extends to any symmetrizable
Kac-Moody algebra $\g$.  In particular, if $\g$ has symmetric Cartan
matrix, then there exists a tensor product quiver variety whose
irreducible components can be given the structure of a tensor
product crystal. There then exists a map $X \mapsto X^\dag$, which
generalizes the map defined above.  One can show that, in the case
of two factors, this map corresponds to the crystal commutor.  For
three factors, the compositions $\sigma^c_{B_{\lambda_1},
B_{\lambda_3} \otimes B_{\lambda_2}} \circ \left(\id_{B_{\lambda_1}}
\otimes \sigma^c_{B_{\lambda_2} \otimes B_{\lambda_3}}\right)$ and
$\sigma^c_{B_{\lambda_2} \otimes B_{\lambda_1}, B_{\lambda_3}} \circ
\left( \sigma^c_{B_{\lambda_1}, B_{\lambda_2}} \otimes
\id_{B_{\lambda_3}} \right)$ both correspond (on highest weight
elements) to the map $X \mapsto X^\dag$ and are therefore equal.
Thus the commutor satisfies the cactus relation. When $\g$ is
symmetrizable but with non-symmetric Cartan matrix, one can use a
well-known folding argument to obtain the same result from the
symmetric case. We therefore have the following theorem.
\begin{theo}[{\cite[Theorem~6.4]{Sav08a}}]
For a symmetrizable Kac-Moody algebra $\g$, the category of
$\g$-crystals is a coboundary monoidal category with cactus commutor
$\sigma^c$.
\end{theo}
This generalizes the previously known result for $\g$ of finite
type.  We refer the reader to \cite{Sav08a} for details.

%%%%%%%%%%%%%%%%%%%%%%%%%%%%%%%%%%%%%%%%%%%%%%%%%%%%%%%%%%%%%%%%%%%%%%%

\bibliographystyle{abbrv}
\bibliography{biblist}

\end{document}